\documentclass[letterpaper,12pt]{article}

\usepackage{color}
\usepackage{amssymb}
\usepackage{amsmath}
\usepackage{amsthm}
\usepackage{amscd}
\usepackage{mathtools}
\usepackage[all]{xy}
\usepackage{tikz-cd}

\usepackage{appendix}

\usepackage{hyperref}
\hypersetup{colorlinks=true,linkcolor=blue,anchorcolor=blue,citecolor=blue}
\usepackage{cleveref}

\usepackage{bm}

\newtheorem{theo}{Theorem}[section]
\newtheorem{prop}[theo]{Proposition}
\newtheorem{lem}[theo]{Lemma}
\newtheorem{cor}[theo]{Corollary}
\newtheorem{defi}[theo]{Definition}
\newtheorem{rem}[theo]{Remark}
\newtheorem{conj}[theo]{Conjecture}

\newtheorem{theon}{Theorem}
\newtheorem{corn}[theon]{Corollary}

\def \fppf {{\rm fppf}}
\def \Cl {{\rm Cl}}
\def \dR {{\rm dR}}
\def \cris {{\rm cris}}
\def \s {{\rm s}}
\def \p {{\rm p}}

\def \Br {{\rm{Br}}}

\def \Ga {{\Gamma}}

\def \Pic {{\rm {Pic}}}

\def \Gal {{\rm{Gal}}}

\def \dim {{\rm{dim}}}

\def \End {{\rm {End}}}

\def\ov{\overline}

\def \Z {{\mathbb Z}}
\def \Q {{\mathbb Q}}
\def \F {{\mathbb F}}

\def\G{{\mathbb G}}

\def\lra{\longrightarrow}

\def\H{{\rm H}}

\def\Zar{{\rm Zar}}

\def\X{{\cal X}}

\def\NS{{\rm NS\,}}

\def\O{{\cal O}}

\def\Ga{\Gamma}

\def\et{{\rm{\acute et}}}

\newcommand{\bthe}{\begin{theo}}
\newcommand{\ble}{\begin{lem}}
\newcommand{\bpr}{\begin{prop}}
\newcommand{\bco}{\begin{cor}}
\newcommand{\bde}{\begin{defi}}
\newcommand{\brem}{\begin{rem}}
\newcommand{\bconj}{\begin{conj}}
\newcommand{\ethe}{\end{theo}}
\newcommand{\ele}{\end{lem}}
\newcommand{\epr}{\end{prop}}
\newcommand{\eco}{\end{cor}}
\newcommand{\ede}{\end{defi}}
\newcommand{\erem}{\end{rem}}
\newcommand{\econj}{\end{conj}}

\newcommand{\isomto}{\overset{\cong}{\longrightarrow}}

\title{Boundedness of the $p$-primary torsion of the Brauer groups of K3 surfaces}
\author{Christopher D.~Lazda\footnote{Department of Mathematics, Harrison Building, University of Exeter, EX4 4QF, United Kingdom. Email: \url{c.d.lazda@exeter.ac.uk}}\; and Alexei N.~Skorobogatov\footnote{Department of Mathematics, South Kensington Campus, Imperial College London, SW7 2AZ, United Kingdom
 \  \ and \  \ 
Institute for the Information Transmission Problems,
Russian Academy of Sciences,
Moscow, 127994, Russia. Email: \url{a.skorobogatov@imperial.ac.uk}}}

\date{\today}
\begin{document}
\maketitle

\begin{abstract}
\noindent We prove that the transcendental Brauer group of a K3 surface $X$ over a finitely generated field $k$
is finite, unless $k$ has positive characteristic $p$ and $X$ is supersingular, in which case it is 
annihilated by $p$.
\end{abstract}

\section*{Introduction}

Let $k$ be a field, and $X$ a smooth, projective, geometrically connected variety over $k$. If $\bar k$ is an algebraic closure of $k$, and $\overline{X}$ is the base change of $X$ to $\bar{k}$, we let
\[ \Br(\overline{X})^k:={\rm Im}[\Br(X)\to\Br(\overline{X})]\]
denote the transcendental Brauer group of $X$. If $k$ is a finitely generated field of characteristic $0$, and $X$ is a K3 surface, it was proved in \cite{SZ08} that this group is finite. Similarly, if $k$ is finitely generated of characteristic $p>0$, and $X$ is again a K3 surface, then it was proved in \cite{SZ15} (for odd $p$) and \cite{Ito18} (for $p=2$) that the prime-to-$p$ torsion subgroup $\Br(\overline{X})^k(p')\subset \Br(\overline{X})^k$ is finite.

In general, one cannot expect the $p$-primary torsion subgroup $\Br(\overline{X})^k\{p\}\subset {\Br}(\overline{X})^k$ to be finite, see \cite[\S 4, Example 1]{Sk}. However, one might still ask if this group has finite exponent, that is, if it is annihilated by some fixed power of $p$. Moreover, the examples where $\Br(\overline{X})^k\{p\}$ is infinite are all supersingular, so one might ask whether $\Br(\overline{X})^k\{p\}$ is finite whenever $X$ is of finite height. Inspired by the analogous result for abelian varieties proved in \cite{D'A}, our main goal in this article is to give an affirmative answer to these questions.

\begin{theon} \label{theo: main} Suppose that $k$ is finitely generated of characteristic $p>0$, and that $X$ is a K3 surface.
\begin{enumerate} \item[{\rm (1)}] \label{num: 1} If $X$ is of finite height $h$, then $\Br(\overline{X})^k\{p\}$ is a finite abelian group with at most $22-2h-\rho$ generators, 
where $\rho$ is the Picard number of $\ov X$.
\item[{\rm (2)}] \label{num: 2} If $X$ is supersingular, then $\Br(\overline{X})\cong (\bar k,+)$ is 
annihilated by $p$. 
\end{enumerate}
In particular, $\Br(\overline{X})^k\{p\}$ is of finite exponent.
\end{theon}

Note than when $k$ is a finite field of characteristic $\neq2$, a more refined version of this result can be found in \cite{Mil75}.

In fact, modulo the Tate conjecture, (2) is well-known, and essentially goes back to Artin \cite{Art74b}. All we do here is observe that the Tate conjecture is now known for K3 surfaces, thanks to work of Charles \cite{Cha13}, Maulik \cite{Mau14}, Madapusi Pera \cite{MP15}, Kim--Madapusi Pera \cite{KMP16} and Ito--Ito--Koshikawa \cite{IIK21}. Our proof of (1) is also based on the proof of the Tate conjecture for K3 surfaces given in \cite{MP15} and \cite{KMP16}, and in particular the arithmetic properties of the Kuga--Satake construction proved there. Our main observation is that the crystalline incarnation of Kuga--Satake descends from the perfection of $k$ to give a map of $F$-isocrystals over $k$ itself (see Theorem \ref{theo: KS} below). Ultimately, this allows us to reduce Theorem \ref{theo: main} to the crystalline Tate conjecture for abelian varieties proved by de\thinspace Jong in \cite{dJ98}. 

Combined with the above-mentioned results on the prime-to-$p$ torsion, we obtain the following 
corollary. Recall that the {\em characteristic exponent} of $k$ is 1 if ${\rm char}(k)=0$,
and $p$ if ${\rm char}(k)=p>0$.

\medskip

\begin{corn} Suppose that $k$ is finitely generated of characteristic
exponent $p$, and that $X$ is a K3 surface. Then $\Br(\overline{X})^k$ is finite unless $p>1$ and $X$ is supersingular, in which case $\Br(\overline{X})^k$ is annihilated by $p$.
\end{corn}

\noindent{\bf Question.} Suppose that $k$ is an {\em infinite} finitely generated field
of characteristic $p>0$, and that $X$ is a {\em supersingular} K3 surface. Is the group $\Br(\overline{X})^k$
infinite?

\medskip

The answer to this question is positive when $X$ is the Kummer surface attached to 
the self-product of a supersingular elliptic curve, and $p\neq 2$ \cite[\S 4, Example 1]{Sk}. In general,
for an infinite field $k$, the group $\Br(\overline{X})^k$ becomes infinite
after replacing $k$ by a finite extension (see Remark \ref{rem} below).

\medskip

Let $k^\s$ be the separable closure of $k$ in $\bar k$, and let $\Ga=\Gal(k^\s/k)$.
Write $X^\s=X\times_k k^\s$.
The natural map $\Br(X^\s)\to\Br(\overline{X})$ is injective because 
the Picard scheme of a K3 surface is smooth, see \cite[Theorem 5.2.5 (ii)]{CTS21}, \cite[Corollary 3.4]{D'A}.
It is known that the cokernel of the natural map $\Br(X)\to\Br(X^\s)^\Ga$ is the direct sum of a
finite group and a $p$-torsion group of finite exponent \cite[Theorem 5.4.12]{CTS21}.
Thus we obtain the following statement.

\begin{corn} Suppose that $k$ is finitely generated of characteristic
exponent $p$, and that $X$ is a K3 surface.
Then $\Br(X^\s)^\Ga$ is the direct sum of a
finite group and a $p$-torsion group of finite exponent.
\end{corn}

When the first version of this paper was written, the authors became aware of the preprint
of Z.~Li and Y.~Qin that gives a different proof of Corollary B. Their result applies to arbitrary smooth
projective varieties for which the Tate conjecture for divisors is known, see 
\cite[Theorem 1.1, Corollary 1.2]{LQ24}.

The authors are very grateful to the referees for their careful reading of the paper, for their 
suggestions, and for spotting a gap in the first version of the paper.

\subsection*{Notation and conventions}

\begin{itemize}
\item For a given field $k$ of characteristic $p>0$, we will let $\bar k$ be an algebraic closure of $k$, $k^\s$ the separable closure of $k$ in $\bar k$,
and $k^\p$ the perfect closure of $k$ in $\bar k$ 
(the maximal inseparable extension of $k$ in $\bar k$). We write $\Ga=\Gal(k^\s/k)$.
\item A variety over $k$ will mean a separated $k$-scheme of finite type.
For a variety $X$ over $k$ we will write $X^\s=X\times_k k^\s$, $X^\p:=X\times_k k^\p$ and $\overline{X}:=X\times_k \bar{k}$.
\item A subset $\{x_i\}\subset k$ is called a $p$-{\em basis} of $k$ if the monomials $\prod x_i^{e_i}$,
where $0\leq e_i\leq p-1$, form a basis of $k$ as a vector space over $k^p\subset  k$, see
\cite[\href{https://stacks.math.columbia.edu/tag/07P1}{Definition 07P1}]{stacks-project}. Any field of characteristic $p$ admits a $p$-basis.
\item Unadorned tensor products will be over $\Z$.
\item For any scheme $S$ we set
$$ \H^i_\fppf(S,\Z_p(1)) := \lim_n \H^i_\fppf(S,\mu_{p^n}). $$
We have followed the definition given in \cite[Chapitre II, \S5]{Ill} in the case when $S$ is smooth and proper over an algebraically closed field. Outside this case, the above definition is not really the `correct' one,\footnote{The `correct' definition is $\H^i({\bf R}\!\lim_n {\bf R}\Gamma(S,\mu_{p^n}))$, which always maps to the above group, and is isomorphic to it if $\lim_n^{(1)}\H^{i-1}(S,\mu_{p^n})=0$.} however, it will be good enough for our purposes here.
\item For any abelian group $A$ we denote by $A\{p\}$ its $p$-primary torsion subgroup, and by $A(p')$ its prime-to-$p$ torsion subgroup. 
\end{itemize}

\section{The geometric Brauer group}

Let $k$ be a field of characteristic $p$, and $X$ a smooth and proper variety over $k$. 
We denote by $\rho=\dim_\Q(\NS(\overline{X})\otimes\Q)$ the (geometric) Picard number of $X$.
Let $r$ be the dimension of the $\Q_p$-vector space 
$(\H^2_\cris(\overline{X}/W({\bar k}))\otimes \Q)^{F=p}$.
Using results of  Illusie, it is shown in \cite[Theorem A.1]{Pet24} that there is
an isomorphism $\Br(\overline{X})\{p\}\cong (\Q_p/\Z_p)^{\oplus(r-\rho)}\oplus B$,
where $B$ is annihilated by a power of $p$. For K3 surfaces, the following
precise result is known (see for example \S2.G of \cite{BY23}).

\bthe Let $X$ be a K3 surface over $k$. If $X$ has finite height $h$, then 
$\Br(\overline{X})\{p\}\cong (\Q_p/\Z_p)^{\oplus(22-2h-\rho)}$. If $X$ is supersingular, then $\Br(\overline{X})\cong\bar k$. 
\ethe

\begin{proof} We include a proof for the reader's convenience. We will assume that $k=\bar{k}$ is algebraically closed (thus $X=\overline{X}$), and we set $W=W({\bar k})$. 

For any smooth and proper variety $X$ over $k$ and any
$i,j\geq 0$ we have a short exact sequence \cite[Corollaire IV.3.5 (a)]{IR83}
$$
0\to \H^j(X,W\Omega^i_{X,{\rm log}})\to \H^j_{\Zar}(X,W\Omega^i_X)\xrightarrow{1-F}\H^j_{\Zar}(X,W\Omega^i_X)\to 0,
$$
where $W\Omega^i_X$ is the sheaf of de Rham--Witt differential forms
with its semi-linear Frobenius $F$ \cite[I.2.E]{Ill}. By definition,
$$\H^j(X,W\Omega^i_{X,{\rm log}}):=\varprojlim \H^j_\Zar(X,W_n\Omega^i_{X,{\rm log}}),$$
where $W_n\Omega^i_{X,{\rm log}}$ is a subsheaf of `logarithmic' forms in $W_n\Omega^i_{X}$.
In view of a canonical isomorphism 
$\H^i(X,\Z_p(1))\cong \H^{i-1}(X,W\Omega^1_{X,{\rm log}})$  \cite[\S II.5]{Ill},
we obtain an isomorphism
$$\H^i(X,\Z_p(1))\cong \H^{i-1}(X,W\Omega^1_{X})^{F=1}.$$

We also have an exact sequence
\begin{equation}\label{eqn: flat mupn} 0 \to \H^2_\fppf(X,\Z_p(1))/p^n \to\H^2_\fppf(X,\mu_{p^n}) \to \H^3_\fppf(X,\Z_p(1))[p^n] \to 0 \end{equation}
together with the exact sequence
\begin{equation} \label{eqn: kummern} 0 \to {\rm Pic}(X)/p^n\to \H^2_\fppf(X,\mu_{p^n}) \to  \Br(X)[p^n]\to 0 \end{equation}
deduced from the Kummer exact sequence in the fppf topology. 

Now let $X$ be a K3 surface. As explained in \cite[Theorem 4.8]{Lie16}, $X$ is Shioda supersingular (that is, $\rho=22$)
if and only if it is Artin supersingular (that is, the $F$-crystal $\H^2_\cris(X/W)$ is purely of slope 1), thanks to the Tate conjecture for K3 surfaces \cite{Cha13,Mau14,MP15,KMP16,IIK21}. (Note that while \cite[Theorem 4.8]{Lie16} is stated in odd characteristic, this restriction comes from the fact that the Tate conjecture for K3 surfaces was only known in odd characteristic when \emph{loc. cit.}~was written. Given the Tate conjecture for K3 surfaces in characteristic $2$, the proof of \cite[Theorem 4.8]{Lie16} works verbatim.)

If $X$ is supersingular, then the map ${\rm Pic}(X)\otimes \Z_p \to \H^2_\fppf(X,\Z_p(1))$ is an isomorphism by \cite[Corollaries 1.6 and 1.7]{Ogu79} (for the missing piece of the latter, see \cite[Chapitre II, Th\'eor\`eme 5.14]{Ill}). We therefore deduce an isomorphism
\[ \Br(X)[p^n] \stackrel{\sim}\lra \H^3_\fppf(X,\Z_p(1))[p^n]\cong \bar k, \]
see \cite[Chapitre II, \S7.2(b)]{Ill} for the second isomorphism.

An entirely similar argument using the $\ell$-adic Kummer sequence shows that $\Br(X)\{\ell\}=0$ for all $\ell\neq p$, and hence $\Br(X)$ is $p$-torsion.

Now suppose that $X$ has finite height $h$. Then $\H^2(X,W\O_{X})$ is finitely generated 
over $W$, since it is the Dieudonn\'e module of the $p$-divisible group ${\widehat{\Br}}(X)$ \cite[Corollary 4.3]{AM77}. In particular, the slope spectral sequence
\[ E_1^{i,j}= \H^j(X,W\Omega^i_{X})\implies \H^{i+j}(X,W\Omega^\bullet_{X})  \]
degenerates at $E_1$ by \cite[Chapitre II, Corollaire 3.14]{Ill}, and thus each $\H^j(X,W\Omega^i_{X})$ is finite free over $W$ since each $\H^{i+j}(X,W\Omega^\bullet_{X})$ is. It then follows that all $\H^i_\fppf(X,\Z_p(1))$
are finite free over $\Z_p$ by the argument of \cite[Lemma A.2]{Pet24}. 
In particular, we deduce from \eqref{eqn: flat mupn} that
\[ \underset{n}{\rm colim}\;\H^2(X,\mu_{p^n}) \isomto \H^2(X,\Z_p(1))\otimes_{\Z_p} \Q_p/\Z_p \]
is of the form $(\Q_p/\Z_p)^{\oplus n}$ for some $n\geq0$. 
By \cite[Th\'eor\`eme II.5.5 (5.5.3)]{Ill} we have an isomorphism of $\Q_p$-vector spaces
$$\H^i(X,\Z_p(1))\otimes\Q\cong(\H^i_\cris(X/W)\otimes \Q)^{F=p}.$$
Thus $n$ is the multiplicity of the slope 1 part of $\H^2_\cris(X/W)$, so that 
$n=22-2h$, cf.~\cite[Chapitre II, \S7.2(a)]{Ill}. 
The Kummer exact sequences \eqref{eqn: kummern} then fit together to give a surjection
\[ \underset{n}{\rm colim}\; \H^2(X,\mu_{p^n})  \twoheadrightarrow \Br(X)\{p\} \] with kernel $(\Q_p/\Z_p)^{\oplus \rho}$.
Thus there is an exact sequence
\[ 0 \to (\Q_p/\Z_p)^{\oplus \rho} \overset{\alpha}{\to} (\Q_p/\Z_p)^{\oplus(22-2h)}\to \Br(X)\{p\} \to 0.  \]
The Pontryagin dual 
of $\alpha$ is a surjective map from $\Z_p^{\oplus(22-2h)}$ to $\Z_p^{\oplus \rho}$, whose kernel is therefore isomorphic to $\Z_p^{\oplus(22-2h-\rho)}$. By double duality, we see that
\[ \Br(X)\{p\}\cong {\rm Hom}(\Z_p^{\oplus(22-2h-\rho)},\Q_p/\Z_p)=(\Q_p/\Z_p)^{\oplus(22-2h-\rho)},\]
as required.
\end{proof}

This immediately proves Theorem \ref{theo: main} in the supersingular case. 

To show that Theorem \ref{theo: main} holds in the finite height case, it is enough to show that the $p$-adic Tate module $T_p(\Br(\overline{X})^k)$ is zero. If $k'/k$ is a finite field
extension, then the obvious injection $\Br(\overline{X})^k\to \Br(\overline{X})^{k'}$ induces an injection $T_p(\Br(\overline{X})^k)\to T_p(\Br(\overline{X})^{k'})$. Thus, we are free to replace $k$ by a finite extension, so we can assume that $X(k)\neq\emptyset$ and that ${\rm Pic}(X)={\rm Pic}(\overline{X})$.

The exactness of the Kummer sequence in the fppf topology
gives rise to a commutative diagram
\begin{equation}\begin{split}
\label{eqn: kummer} \xymatrix{0\ar[r]&\Pic(\overline{X})\otimes\Z_p\ar[r]&\H^2_\fppf(\overline{X},\Z_p(1))\ar[r]&
T_p(\Br(\overline{X}))\ar[r]&0\\
0\ar[r]&\Pic(X)\otimes\Z_p\ar[r]\ar[u]&\H^2_\fppf(X,\Z_p(1))\ar[r]\ar[u]&
T_p(\Br(X))\ar[r]\ar[u]&0}
\end{split}
\end{equation}
with exact rows. We let $T_p(\Br(\overline{X}))^k$ denote the image of the right hand vertical map of \eqref{eqn: kummer}.

\ble
Suppose $X(k)\neq \emptyset$ and $\Pic(X)=\Pic(\overline{X})$. Then 
the natural map $T_p(\Br(\overline{X}))^k\to T_p(\Br(\overline{X})^k)$ is an isomorphism.
\ele

\begin{proof} The assumption implies that $\H^1(k,\Pic(X^\s))=0$.
The map $\Br(X^\s)\to\Br(\overline{X})$ is injective because the Picard scheme of a K3 surface is 
smooth, see \cite[Theorem 5.2.5 (ii)]{CTS21}, \cite[Corollary 3.4]{D'A}.
Thus, the Hochschild--Serre spectral sequence gives rise to an exact sequence
$$0\to\Br(k)\to \Br(X)\to \Br(\overline{X})^k\to 0$$
which is split by the choice of a $k$-point.
Since the sequence is split, we see that $T_p(\Br(X))\to T_p(\Br(\overline{X})^k)$ is surjective. Since $T_p(\Br(\overline{X})^k)\to T_p(\Br(\ov X))$ is injective, 
we deduce that $T_p(\Br(\overline{X}))^k= T_p(\Br(\overline{X})^k)$ as required.
\end{proof}

\begin{rem} {\rm D'Addezio showed that
$T_p(\Br(\overline{A})^k)=0$ for any abelian variety $A$ over $k$ \cite[Theorem 5.1]{D'A},
by deducing this from $T_p(\Br(\overline{A}))^k=0$. See \cite[Theorem 3.2]{Sk} for a different proof.}
\end{rem}

We now let $\H^2_\fppf(\overline{X},\Z_p(1))^k$ denote the image of
\[ \H^2_\fppf(X,\Z_p(1))\to \H^2_\fppf(\overline{X},\Z_p(1)). \]
If ${\rm Pic}(X)={\rm Pic}(\overline{X})$, we deduce from \eqref{eqn: kummer} an exact sequence
\begin{equation} \label{eqn: kummer2} 0\to\Pic(X)\otimes\Z_p\to\H^2_\fppf(\overline{X},\Z_p(1))^k\to T_p(\Br(\overline{X}))^k\to 0. \end{equation}
If, in addition, we have $X(k)\neq\emptyset$, then the group appearing on the right can be replaced by $T_p(\Br(\overline{X})^k)$.

\begin{rem} \label{rem}
{\rm In this remark, we explain why ${\rm Br}(\overline{X})^k$ is potentially infinite whenever $X$ is supersingular, and $k$ is infinite. First, we assume that $X(k)\neq\emptyset$ and ${\rm Pic}(\overline{X})={\rm Pic}(X)$. This implies that ${\rm Br}(X)\cong {\rm Br}(k)
\oplus {\rm Br}(\overline{X})^{k}$, so it is enough to show that ${\rm Br}(X)[p]/{\rm Br}(k)[p]$ is potentially infinite. We next consider the exact sequence
\[ 0\to {\rm Pic}(X)/p \to \H^2_\fppf(X,\mu_p) \to {\rm Br}(X)[p]\to 0. \]
Since ${\rm Pic}(X)/p$ is finite, it is enough to show that $\H^2_\fppf(X,\mu_p)/\H^2_\fppf(k,\mu_p)$ is potentially infinite. We now consider the Leray spectral sequence for the morphism
\[ f\colon X_{\rm fppf}\to {\rm Spec}(k)_{\rm fppf} \]
of flat sites. Note that $f_*\mu_p=\mu_p$, and that ${\bf R}^1f_*\mu_p$ is represented by the $p$-torsion in the Picard scheme of $X$, which is trivial since $X$ is a K3 surface. The existence of a $k$-point means that all maps $\H_\fppf^q(k,\mu_p)\to \H_\fppf^q(X,\mu_p)$ are injective, in particular the differential
\[ \H^0_\fppf(k,{\bf R}^2f_*\mu_p)\to \H_\fppf^3(k,\mu_p) \]
has to be zero, and we deduce an exact sequence
\[ 0 \to \H^2_\fppf(k,\mu_p) \to \H^2_\fppf(X,\mu_p) \to \H^0_\fppf(k,{\bf R}^2f_*\mu_p)\to 0. \]
Now, it is shown in \cite[Proposition 2.2.4]{BL18} that, at least over $\bar{k}$, the sheaf ${\bf R}^2f_*\mu_p$ is represented by a smooth group scheme with connected component isomorphic to $\G_a$. Since ${\bf R}^2f_*\mu_p$ is representable by \cite[Corollary 1.6]{BO21}, it follows that, after making a finite extension of $k$, there exists an embedding $\G_a\hookrightarrow {\bf R}^2f_*\mu_p$, in which case $\H^0_\fppf(k,{\bf R}^2f_*\mu_p)$ must be infinite.}
\end{rem}

\section{Crystalline cohomology}

We continue to let $k$ be an arbitrary field of characteristic $p$. Choose a Cohen ring $C$ for $k$, that is, a complete discrete valuation ring with maximal ideal generated by $p$ and residue field isomorphic to $k$.
Choose a lifting $\sigma\colon C\to C$ of the Frobenius endomorphism of $k$.
If we let $\widehat{C}^\p$ denote the completed (direct) perfection of $C$ along $\sigma$, then $\widehat{C}^\p$ is a Cohen ring for $k^\p$, and is thus canonically isomorphic to $W(k^\p)$.
In particular, $\sigma$ determines an embedding $C\hookrightarrow W(k^\p)$, such that $\sigma$ is the restriction to $C$ of the canonical Witt vector Frobenius on $W(k^\p)$.

Set $C_n:=C/p^n$. If $\{x_i\}_{i\in I}\subset C$ is a lift of a $p$-basis for $k$, then it follows from a more or less direct calculation that the module
\[ \widehat{\Omega}_{C} := \lim_n \Omega_{C_n/\Z} \]
of $p$-adically complete differentials on $C$ is given by $\widehat{\oplus}_{i\in I}\, C\,{\rm d}x_i$. For any smooth and proper $k$-scheme we have the crystalline cohomology groups $\H^q_\cris(X/C)=\lim_n \H^q_\cris(X/C_n)$. By \cite[Summary 7.26]{BO78} these are finitely generated $C$-modules. They are equipped with an integrable connection
\[ \nabla\colon \H^q_\cris(X/C)\to \widehat{\Omega}_C \otimes_C  \H^q_\cris(X/C) \]
by the construction explained in \cite[page 7.16]{BO78}, and a horizontal Frobenius
\[ F\colon \sigma^*\H^q_\cris(X/C)\to \H^q_\cris(X/C)\]
by functoriality. If $X$ is geometrically connected of dimension $d$, then there is a horizontal and Frobenius-equivariant cup-product pairing
\begin{equation} \label{eqn: cup} \H^q_\cris(X/C) \times \H^{2d-q}_\cris(X/C) \to \H^{2d}_\cris(X/C)\cong C(-d), \end{equation}
where the target has the unit connection and twisted Frobenius structure $F=p^d\sigma$.

In the case that $k$ is perfect, Illusie constructs in \cite[Chapitre II, Th\'eor\`eme 1.4]{Ill} a functorial isomorphism
between the crystalline cohomology of $X$ and the (Zariski or \'etale) hypercohomology of the de\thinspace Rham--Witt complex of $X$. This is deduced from a similar isomorphism
\[  \H^q_\cris(X/W_n(k))\isomto \H^q(X,W_n\Omega^\bullet_X) \]
for each $n$, which in fact holds for any smooth variety $X/k$ (not necessarily proper). We will need the following marginal extension of this result. 

\ble \label{lemma: cris dRW} Let $X$ be a smooth variety over a field $k$ of characteristic $p$, which is finitely generated over a perfect subfield $k_0\subset k$. Then, for each $n$, there is an isomorphism
\[\label{eqn: cris dRW} \H^q_\cris(X/W_n(k_0))\isomto \H^q(X,W_n\Omega^\bullet_X) \]
functorial in $X$ and $k$.
\ele

\begin{proof} This is surely well-known to the experts, but we include a proof here for lack of a precise reference. By choosing a separating basis for $k/k_0$ (which exists since $k_0$ is perfect), we may write $k={\rm colim}_{i\in I} R_i$ as a filtered colimit of smooth $k_0$-algebras, with all transition maps $R_i\to R_j$ localisations. Moreover, we may assume that there exists a cofiltered diagram $\{X_i\}_{i\in I}$, with each $X_i$ a smooth $R_i$-scheme, such that $X_j=X_i\times_{R_i} R_j$ whenever $j\geq i$, and $X=X_i\times_{R_i} k$. In particular, $X={\rm lim}_i X_i$.

Then thanks to \cite[Chapitre II, Th\'eor\`eme 1.4]{Ill} there is a natural quasi-isomorphism ${\bf R}\Gamma_\cris(X_i/W_n(k_0))\isomto {\bf R}\Gamma(X_i,W_n\Omega^\bullet_{X_i})$ for each $i$, and we can see that 
\[  {\rm colim}_{i\in I}  {\bf R}\Gamma(X_i,W_n\Omega^\bullet_{X_i}) \isomto {\bf R}\Gamma(X,W_n\Omega^\bullet_{X}) \] by localising on $X$ and using the definition of $W_n\Omega^\bullet$ for smooth $\F_p$-algebras. It is therefore enough to show that the natural map
\[  {\rm colim}_{i\in I}{\bf R}\Gamma_\cris(X_i/W_n(k_0)) \to {\bf R}\Gamma_\cris(X/W_n(k_0)) \]
is an isomorphism. This question is local on $X$, so we can assume that $X={\rm Spec}(A)$ is affine, and so are all $X_i={\rm Spec}(A_i)$. In this case, we lift some fixed $R_{i_0}$ to a smooth $W_n(k_0)$-algebra $\widetilde{R}_{i_0}$, and $A_{i_0}$ to a smooth $\widetilde{R}_{i_0}$-algebra $\widetilde{A}_{i_0}$. Since the transition maps in the system $\{R_i\}_{i\geq i_0}$ are all localisations, we can lift each $R_i$ uniquely to a localisation $\widetilde{R}_i$ of $\widetilde{R}_{i_0}$, and then set $\widetilde{A}_{i}:=\widetilde{A}_{i_0}\otimes_{\widetilde{R}_{i_0}}\widetilde{R}_i$. Thus $\widetilde{A}:={\rm colim}_{i\in I}\widetilde{A}_i$ is a lifting of $A$ to a quasi-smooth $W_n(k_0)$-algebra in the sense of \cite[Chapitre IV, D\'efinition 1.5.1]{Ber74}. We therefore have
\[ {\bf R}\Gamma_\cris({\rm Spec}(A_i)/W_n(k_0)) = {\bf R}\Gamma_\dR(\widetilde{A}_i/W_n(k_0)) \]
for each $i$ by \cite[Chapitre V, Th\'eor\`eme 2.3.2]{Ber74}, as well as a (compatible) comparison morphism
\begin{equation} \label{eqn: comp qs} {\bf R}\Gamma_\cris({\rm Spec}(A)/W_n(k_0)) \to {\bf R}\Gamma_\dR(\widetilde{A}/W_n(k_0)). \end{equation}
Since algebraic de\thinspace Rham cohomology commutes with filtered colimits of rings, it is therefore enough to show that the map \eqref{eqn: comp qs} is an isomorphism.

In the case that $k_0=\F_p$, we can argue as follows. Thanks to \cite[Theorem 1.7 and Proposition 1.4]{Kat91}, it is enough to show that, up to localising on $\widetilde{A}$ (or equivalently $A$), the ring homomorphism $\F_p\to A$ admits a $p$-basis in the sense of \cite[Definition 1.3]{Kat91}. After choosing a separating basis for $k$ over $\F_p$, and local \'etale co-ordinates on $A$, we see that there exists an \'etale map $\F_p(t_1,\ldots,t_n)[X_1,\ldots,X_d]\to A$ for some $n,d$, in which case the elements $t_1,\ldots,t_n,X_1,\ldots,X_n$ clearly form a $p$-basis for $A$ over $\F_p$.

The general case can be handled in exactly the same way, since the proofs of \cite[Theorem 1.7, Proposition 1.4]{Kat91} work equally well with $\Z/p^n\Z$ replaced by $W_n(k_0)$. 
\end{proof}

In particular, in the situation of Lemma \ref{lemma: cris dRW}, the composition of the Teichmuller lift $\O_Y^\times\to W(\O_Y)^\times$ with the map
\begin{align*} {\rm d}\!\log\colon W(\O_Y)^\times&\to W\Omega^1_{Y}\\
u&\mapsto u^{-1}{\rm d}u
\end{align*}
as in \cite[Chapitre I, \S3.23]{Ill} provides a cycle class map
\[ c_1\colon {\rm Pic}(X) \to \H^2(X,W_n\Omega_{X}^\bullet) = \H^2_\cris(X/W_n(k_0)) \to \H^2_\cris(X/C_n) \]
for all $n$. Hence, passing to the limit, there is a cycle class map
\[ c_1 \colon {\rm Pic}(X) \to \H^2_\cris(X/C).  \]
This lands inside the subspace of horizontal sections (because it factors through absolute crystalline cohomology) on which $F=p$ (because $F^*$ acts as multiplication by $p$ on $\Pic(X)$).

More generally, if $S$ is an $\F_p$-scheme, and $f\colon X\to S$ is smooth and proper, then we may consider the relative crystalline cohomology groups ${\bf R}^qf_{\cris *}\O^\cris_{X/\Z_p}$, which are sheaves of $\O^\cris_{S/\Z_p}$-modules on $(S/\Z_p)_\cris$. In the special case when $S={\rm Spec}(k)$, the group $\H^q_\cris(X/C)$ is simply the realisation of ${\bf R}^qf_{\cris *}\O^\cris_{X/\Z_p}$ on the pro-PD-thickening ${\rm Spec}(k)\hookrightarrow {\rm Spf}(C)$.

In general, the ${\bf R}^qf_{\cris *}\O^\cris_{X/\Z_p}$ are not crystals, however, it follows from \cite[Theorem 7.8, Corollary 7.11]{BO78} that they are `derived crystals'. Thus, if ${\bf R}^{q'}f_{\cris *}\O^\cris_{X/\Z_p}$ is flat for each $q' > q$, then ${\bf R}^qf_{\cris *}\O^\cris_{X/\Z_p}$ is indeed a crystal, and this will always be the case if $X/S$ is an abelian scheme or relative K3 surface, since the cohomology of an abelian variety or K3 surface is torsion free. Another application of \cite[Theorem 7.8, Corollary 7.12]{BO78} shows that, under the same flatness hypothesis, the formation of ${\bf R}^qf_{\cris *}\O^\cris_{X/\Z_p}$ commutes with base change. There is also a relative version of the cup-product pairing~\eqref{eqn: cup}. 

If $X/k$ is a K3 surface, then $\H^2_\cris(X/C)$ is free of rank $22$. 
After base changing to $\bar k$, it follows from the construction that the cycle class map $c_1$ factors as
\[ {\rm Pic}(\overline{X}) \to \H^2_\fppf(\overline{X},\Z_p(1)) \isomto\H^2_\cris(\overline{X}/W(\bar k))^{F=p} \]
where the first map comes from the Kummer exact sequence and the second map is an isomorphism \cite[Chapitre II, Theorem 5.14]{Ill}.

If $\xi$ is a primitive polarisation of $X$, of degree $\xi^2=2d$, we write $\Pic(X)^\xi$ for the orthogonal complement to $\xi$, and 
\[ P\H^2_\cris(X/C) := \langle c_1(\xi)\rangle^{\perp} \subset \H^2_\cris(X/C)  \]
for the analogous primitive crystalline cohomology of $(X,\xi)$. This is stable under both $\nabla$ and $F$. We also have the versions $\Pic(\overline{X})^\xi$ and $P\H^2_\cris(\overline{X}/W(\bar k))$ of these groups after base changing to $\bar k$. 

\bde Define the primitive subgroup of $\H^2_\fppf(\overline{X},\Z_p(1))$ by
\begin{align*} P\H^2_\fppf(\overline{X},\Z_p(1)) &:=P\H^2_\cris(\overline{X}/W(\bar k))^{F=p} \\
&= \H^2_\fppf(\overline{X},\Z_p(1)) \cap P\H^2_\cris(\overline{X}/W(\bar k)).
\end{align*}
Similarly, define the primitive subgroup
\[ P\H^2_\fppf(\overline{X},\Z_p(1))^k:= \H^2_\fppf(\overline{X},\Z_p(1))^k\cap P\H^2_\fppf(\overline{X},\Z_p(1)). \]
\ede

\ble \label{lemma: primitive kummer} Suppose that $(X,\xi)$ is a primitively polarised K3 surface with $X(k)\neq\emptyset$ and ${\rm Pic}(X)={\rm Pic}(\overline{X})$. Then \eqref{eqn: kummer2} induces an exact sequence
\[ 0\to\Pic(X)^\xi\otimes\Q_p\to P\H^2_\fppf(\overline{X},\Z_p(1))^k \otimes \Q\to
T_p(\Br(\overline{X})^k)\otimes \Q\to 0 \]
of $\Q_p$-vector spaces.
\ele

\begin{proof}
The only part that is not clear is the claim that the last map is surjective. For this, it is enough to show that ${\rm Pic}(X)+P\H^2_\fppf(\overline{X},\Z_p(1))$ is of finite index in $\H^2_\fppf(\overline{X},\Z_p(1))$.
To see this, take $x\in \H^2_\fppf(\overline{X},\Z_p(1))$.
By Frobenius compatibility of the cup-product pairing, we see that $x\cup \xi\in \H^4(\overline{X}/W(\bar k))^{F=p^2}=\Z_p$. 
Hence $\xi^2\cdot x-(x\cup \xi)\xi$ lies in $P\H^2_\fppf(\overline{X},\Z_p(1))$, and we are done.
\end{proof}

If $A/k$ is an abelian variety, then $\H^1_\cris(A/C)$ is a free $C$-module of rank $2\,\dim A$, and can be identified with the Dieudonn\'e module of the $p$-divisible group $A[p^\infty]$ by \cite[(3.3.7.2)]{BBM82}. In particular, it follows from \cite[Th\'eor\`eme 4.1.1]{BM90} that there is a commutative diagram
\begin{equation} \label{eqn: square}  
\begin{split}
\xymatrix{  {\rm End}(A[p^\infty]) \ar[d] \ar[r]^-{\cong}& {\rm End}_{C}(\H^1_\cris(A/C))^{\nabla=0,F=1} \ar[d] \\ {\rm End}(\overline{A}[p^\infty]) \ar[r]^-{\cong} & {\rm End}_{W(\bar k)}(\H^1_\cris(\overline{A}/W(\bar k))^{F=1}  }
\end{split}
\end{equation}
in which both horizontal arrows are isomorphisms.

\section{The Kuga--Satake construction}

Let $X$ be a K3 surface over a field $k$ of characteristic $p$, together with a primitive polarisation $\xi$ of degree $\xi^2=2d$. We consider the lattice
$$L_d:={\rm E}_8(-1)^{\oplus 2}\oplus {\rm U}^{\oplus 2}\oplus \Z(-2d),$$
where ${\rm U}$ is the hyperbolic plane, and $(m)$ denotes multiplying the bilinear form by $m$. Recall that the pair $(X,\xi)$ is called {\em superspecial} if $c_1(\xi)\in F^2\H^2_\dR(X/k)$. In particular, this implies that $X$ is supersingular.  

\bthe \label{theo: KS} Let $(X,\xi)$ be a non-superspecial. Up to replacing $k$ by a finite separable extension, there exists an abelian variety $A/k$, equipped with an action of the Clifford algebra ${\rm Cl}(L_d)$, and
\begin{itemize}
\item a horizontal, Frobenius-equivariant embedding
\begin{equation} P\H^2_\cris(X/C) \otimes \Q(1) \hookrightarrow 
\End_{\Cl(L_d)}(\H^1_\cris(A/C)) \otimes \Q; \label{embKS2} \end{equation}
\item for all $\ell\neq p$, a $\Ga$-equivariant embedding
\[ P\H^2_\et(X_{k^\s},\Q_\ell(1))\hookrightarrow 
\End_{\Cl(L_d)}(\H^1_\et(A_{k^\s},\Q_\ell)). \label{embKS3} \]
\end{itemize}
\ethe

\begin{proof} If $p\neq 2$, then the integral version of this statement is almost \cite[Theorem 5.17]{MP15}. The missing point is to descend the $p$-adic cohomological comparison from $W(k^\p)$ to $C$, but this is a simple consequence of \cite[Theorem 5.14]{MP15}. Indeed, it follows from \emph{loc. cit.} that, after taking a suitable \'etale cover $T\rightarrow {\sf M}^\circ_{2d}$ of the moduli space of primitively polarised K3 surfaces, there is an embedding
\begin{equation}\label{eqn: KS cris} {\bm P}_\cris \hookrightarrow {\bm H}_\cris \otimes {\bm H}^\vee_\cris(-1) \end{equation}
of $F$-crystals on $T_{\F_p}$, where ${\bm P}_\cris$ is the relative primitive (crystalline) cohomology of the universal K3 surface, and ${\bm H}_\cris$ is the relative (crystalline) cohomology of the Kuga--Satake abelian scheme. The embedding \eqref{embKS2} is then obtained by simply pulling back via ${\rm Spec}(k)\to T$ and inverting $p$. That it lands inside the subspace of ${\rm Cl}(L_d)$-equivariant endomorphisms can be checked after base changing to $k^\p$. In fact, since we only need a comparison for \emph{rational} crystalline cohomology, and $(X,\xi)$ is non-superspecial, it would suffice to apply \cite[Proposition 5.6, Lemma 5.9]{MP15} (see the argument outlined below in the case $p=2$). 

If $p=2$ then we may obtain the required map in rational crystalline cohomology by noting that the proofs of \cite[Proposition 5.6, Lemma 5.9]{MP15} work verbatim in this case. We explain them here for the reader's convenience. Since $(X,\xi)$ is non-superspecial, it corresponds to a smooth point of the moduli space ${\sf M}_{2d,\Z_{(2)}}$, which then lifts to a point $x$ in the domain $\tilde{{\sf M}}_{2d,\Z_{(2)}}^{\rm sm}$ of the $2$-adic Kuga--Satake map $\iota^{\rm KS}_{\Z_{(2)}}$ constructed in \cite[Appendix A]{KMP16}. In particular, we may choose an \'etale   map $T\to \tilde{{\sf M}}^{\rm sm}_{2d,\Z_{(2)}}$, with $T$ a scheme smooth over $\Z_{(2)}$, such that the Kuga--Satake abelian scheme is defined over $T$ and $x$ is in the image of $T$. Let $\bm{P}_{\dR,\Q}(1)$ and $\bm {H}_{\dR,\Q}$ denote the (primitive) relative de\thinspace Rham cohomology of the universal K3 surface and Kuga--Satake abelian scheme over $T_{\Q}$, respectively. Then, thanks to \cite[Proposition 5.6]{MP15}, there is an injection
\[ \alpha_{\dR,\Q}\colon {\bm P}_{\dR,\Q}(1) \hookrightarrow {\bm H}_{\dR,\Q} \otimes {\bm H}^\vee_{\dR,\Q} \]
of modules with connection on $T_{\Q}$ whose pullback via any point $s\colon {\rm Spec}(F)\to T_{\Q}$ comes from an injection
\begin{equation} \label{eqn: motive point} \alpha_s\colon {\bm P}_s(1)\hookrightarrow {\bm H}_s\otimes {\bm H}_s^\vee \end{equation}
of motives for absolute Hodge cycles over $F$. Let $\widehat{T}$ denote the $2$-adic completion of $T$ as a formal scheme over $\Z_2$. Then, via the usual comparison between crystalline and de\thinspace Rham cohomology, we may view (the restriction of) $\alpha_{\dR,\Q}$ (to the rigid analytic variety $\widehat{T}_{\Q_2}$) as an injection
\[ \alpha_{\cris,\Q}\colon {\bm P}_{\cris}\otimes\Q(1)  \hookrightarrow {\bm H}_{\cris} \otimes {\bm H}^\vee_{\cris} \otimes \Q\]
of isocrystals on $T_{\F_2}$. To see that it is an injection of $F$-isocrystals, it is enough to prove compatibility with Frobenius, and it is enough to show this after pullback to any field-valued point $s_0\colon {\rm Spec}(\F) \to T_{\F_2}$. In fact, because the category of isocrystals on a connected scheme, smooth over $\F_2$, is Tannakian, it suffices to do this at one point on each connected component of $T_{\F_2}$, so we can assume that $\F$ is finite. If we take a lift $s\colon {\rm Spec}(W(\F)) \to T$ of $s_0$ to the Witt vectors of $\F$, we therefore need to show that the de\thinspace Rham realisation
\[ \alpha_{s_\Q,\dR} \colon {\bm P}_{s_{\Q},\dR}(1)\hookrightarrow {\bm H}_{s_{\Q},\dR}\otimes {\bm H}_{s_{\Q},\dR}^\vee \]
of \eqref{eqn: motive point} is compatible with Frobenius (via the identifications of both sides with an appropriate crystalline cohomology group). Since $\alpha_{s_\Q,\dR}$ arises from an absolute Hodge cycle on an abelian variety with good reduction, this follows from \cite[Theorem 2.8 and Lemma 2.11]{MP15}. 

The upshot, then, is that we get the $2$-adic, rational analogue of \eqref{eqn: KS cris}, which we may again pullback via the map $x\colon{\rm Spec}(k)\to T$ to obtain \eqref{embKS2}.
\end{proof}

After base changing to $\bar k$, we thus obtain from \eqref{embKS2} an $F$-equivariant embedding
\begin{equation}
P\H^2_\cris(\overline{X}/W(\bar k))\otimes \Q(1) \hookrightarrow 
\End_{\Cl(L_d)}(\H^1_\cris(\overline{A}/W(\bar k))) \otimes \Q. \label{emb1}
\end{equation}
After replacing $k$ by a finite separable extension, we can assume that the natural map
$\End(A)\to\End(\overline{A})$ is an isomorphism (using a theorem of Chow \cite[Theorem 3.19]{Con06}
which says that $\End(A^\s)\to \End(\overline{A})$ is an isomorphism). In this case, let $L(A)\subset\End_{\Cl(L_d)}(A)$ be the subgroup of {\em special endomorphisms}, that is, endomorphisms whose crystalline realisation comes from the cohomology of $X$ via \eqref{emb1}. 

Base change for crystalline cohomology along the extensions $k\to k^\p\to \bar k$ shows that an endomorphism of $A$ is special if and only if its crystalline realisation lies in
 \[ P\H^2_\cris(X/C) \otimes\Q(1) \subset \End_{\Cl(L_d)}(\H^1_\cris(A/C))\otimes \Q, \]
 or equivalently if and only if its crystalline realisation lies in 
 \[ P\H^2_\cris(X^{\p}/W(k^\p)) \otimes\Q(1) \subset \End_{\Cl(L_d)}(\H^1_\cris(A^{\p}/W(k^\p)))\otimes \Q. \]
 As noted in \cite[p.645]{MP15}, the $\ell$-adic realisations of any special endomorphism (for $\ell\neq p$) lies in
 \[ P\H^2_\et(X_{k^\s},\Q_\ell(1))\subset 
\End_{\Cl(L_d)}(\H^1_\et(A_{k^\s},\Q_\ell)),\]
 hence the definition of a special endomorphism given here coincides with that given in \cite[Theorem 5.17 (4)]{MP15}, as well as that given in \cite[p.645]{MP15}.

\bthe \label{theo: KS2} Suppose $(X,\xi)$ is non-superspecial, and $k$ is large enough that $\End(A)=\End(\overline{A})$ and ${\rm Pic}(X)={\rm Pic}(\overline{X})$. Then there is an injection
\[ L(A) \hookrightarrow \Pic(X)^\xi   \]
such that the diagram
\begin{equation} 
\begin{split}
\xymatrix{ P\H^2_\cris(\overline{X}/W(\bar k)) \otimes \Q (1)
 \ar@{^{(}->}[r] &\End_{\Cl(L_d)}(\H^1_\cris(\overline{A}/W(\bar k))\otimes \Q \\
\Pic(X)^\xi\otimes\Q_p \ \ar@{^{(}->}[u] &L(A)\otimes\Q_p\ar@{_{(}->}[l]\ar@{^{(}->}[u]}
\label{emb4}  \end{split}\end{equation}
commutes. 
\ethe

\begin{proof} If $p\neq 2$, this is just the rational version of \cite[Theorem 5.17 (4)]{MP15}, with a stronger hypothesis and weaker conclusion. 

If $p=2$, then the required injection $L(A) \hookrightarrow \Pic(X)^\xi$ is constructed in \cite[Proposition A.15]{KMP16}, and it is shown there that it is compatible with all $\ell$-adic realisations, in the sense that the diagrams
\begin{equation*} \xymatrix{ P\H^2_\et(\overline{X},\Q_\ell(1)) 
 \ar@{^{(}->}[r] &\End_{\Cl(L_d)}(\H^1_\et(\overline{A},\Q_\ell))  \\
\Pic(X)^\xi\otimes\Q_\ell \ \ar@{^{(}->}[u] &L(A)\otimes\Q_\ell\ar@{_{(}->}[l]\ar@{^{(}->}[u]}
\label{emb5}  \end{equation*}
for $\ell\neq2$ all commute.

To see that the $2$-adic analogue \eqref{emb4} of this square commutes, we follow the proof of \cite[Proposition A.15]{KMP16}. Indeed, we can assume that $k={\bar k}$, and we let $W=W({\bar k})$ denote the ring of Witt vectors. Given $f\in L(A)$, it is shown in \cite[Lemma A.13]{KMP16} that there exists a finite extension $K$ of $W[1/2]$ such that the triple $(X,\xi,f)$ lifts to $\O_K$. In other words, there exists a polarised K3 surface $(Y,\xi)$ over $K$ lifting $(X,\xi)$, whose Kuga--Satake variety $B$ is a lift of $A$, in such a way that $f$ lifts to a special endomorphism of $B$. 

 After making a further finite extension of $K$ we will assume that $\Pic(Y)=\Pic(Y_{\overline{K}})$ and ${\rm End}(B)={\rm End}(B_{\overline{K}})$. We let $L'(B)$ denote the set of special endomorphisms of $B$ which reduce modulo $2$ to a special endomorphism of $A$. We then have a diagram
\[ \xymatrix{ \Pic(X)^\xi   & L(A) \ar[l] \\ \Pic(Y)^\xi \ar[u]& L'(B) \ar[u]\ar[l] }\]
which can be seen to commute by, for example, considering the maps into $\ell$-adic cohomology for any prime $\ell\neq 2$. By construction, $f$ is in the image of the right hand map.

 We now consider the following cube:
\[ \xymatrix@C=-25pt{  P\H^2_\cris(X/W) \otimes_{W} K \ar[rr]^-{\alpha_\cris} & & {\rm End}_{{\rm Cl}(L_d)}(\H^1_\cris(A/W)) \otimes_{W} K \\
 & P\H^2_\dR(Y/K)  \ar[rr]^{\alpha_\dR} \ar[ul]  & & {\rm End}_{{\rm Cl}(L_d)}(\H^1_\dR(B/K))\ar[ul] \\
 \Pic(X)^\xi \otimes \Q_2 \ar[uu]\ & & L(A)\otimes\Q_2 \ar[ll]\ar[uu] \\ & \Pic(Y)^\xi \otimes \Q_2  \ar[uu] \ar[ul] & & L'(B)\otimes \Q_2\ar[uu]\ar[ll] \ar[ul]  } \]
By definition, the top face of this cube commutes. Moreover, by compatibility of the cycle class map with the comparison between crystalline and de\thinspace Rham cohomology 
(combine \cite[Corollary 3.7]{BO83} with \cite[Chapitre III, \S2.1]{Gro85}), the left hand face commutes. We have just shown that the bottom face commutes, and it is a straightforward consequence of functoriality of cohomology that the right hand face commutes. To see that the front face commutes, we note that by \cite[Proposition A.15]{KMP16} the diagram
\[  \xymatrix{ P\H^2_\et(Y_{\overline{K}},\Q_2(1)) 
 \ar[r]^-{\alpha_2 } &\End_{\Cl(L_d)}(\H^1_\et(B_{\overline{K}},\Q_2))  \\
\Pic(Y)^\xi\otimes\Q_2 \ \ar[u] &L(B)\otimes\Q_2\ar[l]\ar[u]}  \]
commutes, so we may apply Fontaine's functor ${\bf D}_\dR:=(-\otimes_{\Q_2} B_\dR)^{{\rm Gal}(\overline{K}/K)}$ and use \cite[Lemma 2.11]{MP15} to identify ${\bf D}_\dR(\alpha_2)$ with $\alpha_{\dR}$.

By construction, $f\in L(A)$ is in the image of $L'(B) \to L(A)$. Thus, by the commutativity of the `other' five faces of the above cube, it follows that the images of $f$ under the two possible ways
\[ L(A)\otimes \Q_2 \rightrightarrows \End_{\Cl(L_d)}(\H^1_\cris(A/W))\otimes \Q\]
of following round the square \eqref{emb4} coincide. Since $f$ was arbitrary, we see that these two maps are in fact equal.
\end{proof}

In particular, passing to $F$-invariants in the top row of \eqref{emb4}, we obtain the commutative diagram
\begin{equation}\begin{split} \xymatrix{ P\H^2_\fppf(\overline{X},\Z_p(1)) \otimes \Q 
 \ar@{^{(}->}[r] &\End_{\Cl(L_d)}(\overline{A}[p^\infty])\otimes \Q \\
\Pic(X)^\xi\otimes\Q_p \ \ar@{^{(}->}[u] &L(A)\otimes\Q_p\ar@{_{(}->}[l]\ar@{^{(}->}[u]}
\label{emb3} 
\end{split}\end{equation}

\section{Completion of the proof}

In the previous section we constructed, for any non-superspecial, primitively polarised K3 surface $(X,\xi)$ over a field $k$ of charactersitic $p>0$, the commutative diagram (\ref{emb3}). We need one more preparatory result before we can prove Theorem \ref{theo: main}.

\bpr Suppose that $k$ is finitely generated, and $(X,\xi)$ is non-superspecial. Then the image of the map
\[ P\H^2_\fppf(\overline{X},\Z_p(1))^k  \to  P\H^2_\cris(\overline{X}/W(\bar k)) \]
is contained within $ P\H^2_\cris(X/C)^{\nabla=0,F=p}$. 
\epr

\begin{proof} It is enough to show that the composite map
\[ \H^2_\fppf(X,\Z_p(1))\to \H^2_\fppf(\overline{X},\Z_p(1))\to \H^2_\cris(\overline{X}/W(\bar k))   \]
factors through $\H^2_\cris(X/C)^{\nabla=0}$. In fact, we may work modulo $p^n$ and show that the composite map
\[ \H^2_\fppf(X,\mu_{p^n})\to \H^2_\fppf(\overline{X},\mu_{p^n})\to \H^2_\cris(\overline{X}/W_n(\bar k)) \]
factors through $\H^2_\cris(X/C_n)^{\nabla=0}$, for all $n$.

Let $k_0$ denote the algebraic closure of $\F_p$ inside $k$, and consider the absolute crystalline cohomology group $\H^2_\cris(X/W_n(k_0))$. Then the map
\[ \H^2_\cris(X/W_n(k_0))\to \H^2_\cris(\overline{X}/W_n(\bar k))\]
factors through $\H^2_\cris(X/C_n)^{\nabla=0}$, so it is enough to produce a commutative diagram
\[ \xymatrix{\H^2_\fppf(X,\mu_{p^n}) \ar[r]\ar[d] & \H^2_\cris(X/W_n(k_0))\ar[d] \\ \H^2_\fppf(\overline{X},\mu_{p^n}) \ar[r]& \H^2_\cris(\overline{X}/W_n(\bar k)). } \]
To do this, we use \cite[Chapitre II, Th\'eor\`eme 1.4]{Ill} and Lemma \ref{lemma: cris dRW} to get identifications
\[ \H^2_\cris(X/W_n(k_0)) \isomto \H^2(X,W_n\Omega^\bullet_X), \quad
\H^2_\cris(\overline{X}/W_n(\bar k)) \isomto \H^2(\overline{X},W_n\Omega^\bullet_{\overline{X}})\]
(note here the use of the hypothesis that $k$ is finitely generated). In fact, since each $W_n\Omega^i$ is a coherent $W_n(\O)$-module, we may replace the Zariski topology here by the \'etale topology. Moreover, by pushing forward from the flat to the \'etale site, we obtain identifications
\begin{align*} \H^2_\fppf(X,\mu_{p^n}) &\isomto \H^1_\et(X,\O_{X}^\times/(\O_{X}^\times)^{p^n}) \\
\H^2_\fppf(\overline{X},\mu_{p^n}) &\isomto \H^1_\et(\overline{X},\O_{\overline{X}}^\times/(\O_{\overline{X}}^\times)^{p^n})
\end{align*}
via which the map $\H^2_\fppf(\overline{X},\mu_{p^n})\to \H^2_\et(\overline{X},W_n\Omega^\bullet_{\overline{X}})$ is induced by 
\[ \O_{\overline{X}}^\times/(\O_{\overline{X}}^\times)^{p^n} \overset{{\rm d}\!\log}{\longrightarrow} W_n\Omega^1_{\overline{X}}.  \]
It is now enough to observe that there is an analogous map
$$\O_{X}^\times/(\O_{X}^\times)^{p^n} \overset{{\rm d}\!\log}{\longrightarrow} W_n\Omega^1_{X}$$
on $X_\et$, and that the commutativity of the diagram
\[ \xymatrix{  \H^1_\et(X,\O_{X}^\times/(\O_{X}^\times)^{p^n}) \ar[r]\ar[d] &\H^2_\et(X,W_n\Omega^\bullet_{X}) \ar[d] \\   \H^1_\et(\overline{X},\O_{\overline{X}}^\times/(\O_{\overline{X}}^\times)^{p^n}) \ar[r] & \H^2_\et(\overline{X},W_n\Omega^\bullet_{\overline{X}}) } \]
is clear.
\end{proof}

\begin{proof}[Proof of Theorem \ref{theo: main}]
We may assume that $X$ is of finite height, and admits a primitive polarisation $\xi$. Set $2d=\xi^2$. Extending $k$ if necessary, we may assume that $X(k)\neq\emptyset$, that ${\rm Pic}(X)={\rm Pic}(\overline{X})$, that the Kuga--Satake variety $A$ is defined over $k$, and that ${\rm End}(A)={\rm End}(\overline{A})$. Thanks to de Jong's $p$-adic Tate conjecture for abelian varieties \cite[Theorem 2.6]{dJ98} we can identify
\[ {\rm End}_{\nabla,F}(\H^1_\cris(A/C))\otimes\Q= \End(A[p^\infty])\otimes \Q=\End(A)\otimes \Q_p.\]
We therefore have the following commutative diagram:
\begin{equation} 
\begin{split}\label{eqn: pent} \xymatrix@C=1.5pt{ &P\H^2_\cris(X/C)^{\nabla=0,F=p}\ar@{^{(}->}[dr] \otimes \Q  &  \\ P\H^2_\fppf(\overline{X},\Z_p(1))^k \otimes \Q 
 \ar@{^{(}->}[ur] & & \End(A) \otimes \Q_p \\
\Pic(X)^\xi\otimes\Q_p \ \ar@{^{(}->}[u] & & L(A)\otimes\Q_p\ar@{_{(}->}[ll]\ar@{^{(}->}[u]} 
\end{split} \end{equation}
We claim that \eqref{eqn: pent} induces an isomorphism
\[  L(A)\otimes \Q_p \isomto P\H^2_\cris(X/C)^{\nabla=0,F=p} \otimes \Q. \]
To prove this, let $k_0$ denote the algebraic closure of $\F_p$ inside $k$. By spreading out the morphism from ${\rm Spec}(k)$ to (a suitable \'etale cover of) the moduli space $\tilde{{\sf M}}^{\rm sm}_{2d,\F_p}$, we can choose a smooth, geometrically connected $k_0$-variety $U$, with function field $k$, such that everything extends over $U$. Concretely, this gives a relative K3 surface $f\colon \X\to U$, a primitive polarisation $\xi\in {\bf Pic}_{\X/U}(U)$ of degree $2d$, an abelian scheme $\pi\colon \mathcal{A}\to U$, and an injection 
\[ {\bm P}{\mathbf R}^2f_{\cris*}\O_{\mathcal{X}/\Z_p}^\cris \otimes \Q(1) \hookrightarrow \End_{\O_{U/\Z_p}^\cris}({\mathbf R}^1\pi_{\cris*}\O_{\mathcal{A}/\Z_p}^\cris \otimes \Q) \]
of $F$-isocrystals on $U/W(k_0)$, which recover $X$, $\xi$, $A$, and the map from Theorem \ref{theo: KS}, respectively, after base changing from $U$ to ${\rm Spec}(k)$. Moreover, base changing the whole situation to any \emph{closed} point $s$ of $U$ recovers the Kuga--Satake correspondence for the (non-superspecial) polarised K3 surface $(\mathcal{X}_s,\xi)$ over the residue field $k_0(s)$. In particular, we have the analogue
\begin{equation} 
\begin{split}
\label{eqn: square2} \xymatrix{  P\H^2_\cris(\mathcal{X}_{s_0}/W(k_0))^{F=p} \otimes \Q \ar@{^{(}->}[r] & {\rm End}(\mathcal{A}_{s_0}) \otimes \Q_p \\ {\rm Pic}(\mathcal{X}_{s_0})^\xi \otimes \Q_p\ar@{^{(}->}[u]& L(\mathcal{A}_{s_0})\otimes \Q_p \ar@{_{(}->}[l]\ar@{^{(}->}[u]   } 
\end{split}\end{equation}
of the diagram \eqref{eqn: pent} over $s_0$. 

After shrinking $U$ if necessary, we may assume that  ${\rm End}(A)={\rm End}_{U}(\mathcal{A})$, and after extending $k_0$ we may assume that there exists a point $s_0\in U(k_0)$ such that ${\rm End}(\mathcal{A}_{s_0})={\rm End}(\mathcal{A}_{\bar{s}_0})$ for any geometric point $\bar{s}_0$ above $s_0$. Every element of
\[ {\rm End}_C(\H^1_\cris(A/C))^{\nabla=0,F=1}\otimes \Q={\rm End}(A)\otimes \Q_p={\rm End}_{U}(\mathcal{A})\otimes \Q_p\]
thus extends to a horizontal and Frobenius equivariant global section of $\End_{\O_{U/\Z_p}^\cris}({\mathbf R}^1\pi_{\cris*}\O_{\mathcal{A}/\Z_p}^\cris \otimes \Q)$. It follows that every horizontal section of $P\H^2_\cris(X/C)$ on which $F=p$ extends to one of ${\bm P}{\mathbf R}^2f_{\cris*}\O_{\mathcal{X}/\Z_p}^\cris$. We therefore obtain compatible (injective!) specialisation maps from the terms in the diagram \eqref{eqn: pent} to the corresponding terms in the diagram \eqref{eqn: square2}, and we can view everything in sight as living  inside ${\rm End}(\mathcal{A}_{s_0})\otimes \Q_p$. It follows from \cite[Theorem 6.4]{MP15} that, for $p\neq 2$,
\begin{equation}\label{eqn: key p}  P\H^2_\cris(\mathcal{X}_{s_0}/W(k_0))^{F=p} \otimes \Q = L(\mathcal{A}_{s_0}) \otimes \Q_p \end{equation}
as subobjects of ${\rm End}(\mathcal{A}_{s_0})\otimes \Q_p$. That the same holds for $p=2$ is essentially observed during the proof of \cite[Theorem A.8]{KMP16}. Indeed, it is shown there that 
\begin{equation} \label{eqn: key l} P\H^2_\et(\mathcal{X}_{\bar{s}_0},\Q_\ell(1))^{{\rm Gal}(\bar{s}_0/s_0)} = L(\mathcal{A}_{s_0}) \otimes \Q_\ell \end{equation} 
for any prime $\ell\neq 2$. Combining $\ell$-independence \cite[Theorem 1]{KM74} with Frobenius semisimplicity for abelian varieties over finite fields shows that the 
left hand sides of \eqref{eqn: key p} and \eqref{eqn: key l} have the same dimension, and combining these then gives the equality in \eqref{eqn: key p}. 

Since $U$ is connected, a horizontal section of $\End_{\O_{U/\Z_p}^\cris}({\mathbf R}^1\pi_{\cris*}\O_{\mathcal{A}/\Z_p}^\cris \otimes \Q)$ lies in ${\bm P}{\mathbf R}^2f_{\cris*}\O_{\mathcal{X}/\Z_p}^\cris \otimes \Q(1)$ if and only if it does so at $s_0$. Thus we see that
\[ L(A) ={\rm End}(A)\cap L(\mathcal{A}_{s_0})\]
inside ${\rm End}(\mathcal{A}_{s_0})$. Finally, we deduce that
\[ P\H^2_\cris(X/C)^{\nabla=0,F=p} \otimes \Q \subset  P\H^2_\cris(\mathcal{X}_{s_0}/W(k_0))^{F=p} \otimes \Q = L(\mathcal{A}_{s_0}) \otimes \Q_p \]
is contained in 
\[ {\rm End}(A) \otimes \Q_p  \cap L(\mathcal{A}_{s_0}) \otimes \Q_p = ({\rm End}(A)\cap L(\mathcal{A}_{s_0})) \otimes \Q_p = L(A)\otimes \Q_p \]
as claimed.

It now follows that the lower horizontal and left hand vertical maps of \eqref{eqn: pent} are isomorphisms. Lemma \ref{lemma: primitive kummer} now shows that $T_p(\Br(\overline{X})^k)\otimes \Q=0$, and hence that $T_p(\Br(\overline{X})^k)=0$, completing the proof.
\end{proof}

\providecommand{\bysame}{\leavevmode\hbox to3em{\hrulefill}\thinspace}

\end{document}